\documentclass{amsart}
\usepackage{amssymb, amsmath}

\newtheorem{theorem}{Theorem}[section]
\newtheorem{proposition}[theorem]{Proposition}

\newtheorem{lemma}[theorem]{Lemma}
\newtheorem{definition}[theorem]{Definition}
\newtheorem{example}[theorem]{Example}
\newtheorem{question}[theorem]{Question}
\numberwithin{equation}{section}

\begin{document}
\title{ $P$-spaces and the Volterra property}

\author{Santi Spadaro}
\address{Department of Mathematics, Faculty of Science and Engineering, York University, Toronto, ON,  M3J 1P3 Canada}
\email{santispadaro@yahoo.com}

\thanks{Supported by an INdAM Cofund outgoing fellowship}

\begin{abstract}
We study the relationship between generalizations of $P$-spaces and Volterra (weakly Volterra) spaces, that is, spaces where every two dense $G_\delta$ have dense (non-empty) intersection. In particular, we prove that every dense and every open, but not every closed subspace of an almost $P$-space is Volterra and that there are Tychonoff non-weakly Volterra weak $P$-spaces. These results should be compared with the fact that every $P$-space is hereditarily Volterra. As a byproduct we obtain an example of a hereditarily Volterra space and a hereditarily Baire space whose product is not weakly Volterra. We also show an example of a Hausdorff space which contains a non-weakly Volterra subspace and is both a weak $P$-space and an almost $P$-space. \end{abstract}


\maketitle

\section{Introduction}


A real-valued function $f$ is called \emph{pointwise discontinuous} if the set of all points where it is continuous is dense. In 1881, eighteen years before Ren\'e-Louis Baire published the Baire category theorem \cite{B}, a twenty years old student of the \emph{Scuola Normale Superiore di Pisa} named Vito Volterra proved that there are no two pointwise discontinous real-valued functions on $\mathbb{R}$ such that the set of all points of continuity of one is equal to the set of all discontinuity points of the other \cite{V} (see also \cite{D}). Volterra's theorem has inspired an interesting generalization of the Baire property.

Given $f: X \to \mathbb{R}$, let $C(f)$ be the set of all continuity points of $f$. 

\begin{definition}
\cite{GP} A topological space $X$ is called \emph{Volterra} (respectively, \emph{weakly Volterra}) if for every pair of pointwise discontinuous functions $f: X \to \mathbb{R}$ and $g: X \to \mathbb{R}$ the set $C(f) \cap C(g)$ is dense in $X$ (respectively, non-empty).
\end{definition}

Thus Volterra's theorem can be rephrased by stating that the real line is a Volterra space.  Gauld and Piotrowski proved the following internal characterization of Volterra and weakly Volterra spaces. Recall that a set is called a \emph{$G_\delta$ set} if it can be represented as a countable intersection of open sets.

\begin{proposition}
\cite{GP} A space is Volterra (respectively, weakly Volterra) if and only if for every pair $G$ and $H$ of dense $G_\delta$ subsets of $X$, the set $G \cap H$ is dense (respectively, non-empty).
\end{proposition}

Recall that a space is Baire if every countable intersection of dense open sets is dense. From the above characterization it's clear that every Baire space is Volterra. The problem of when a Volterra space is Baire has been studied extensively (see \cite{CJ} and \cite{GL}).

This note was inspired by the simple observation that every $P$-space (that is, a space where every $G_\delta$ set is open) is hereditarily Volterra. Weak $P$-spaces and almost $P$-spaces are the two most popular weakenings of $P$-spaces. We compare these properties with the notions of Volterra and weakly Volterra space. We find that every dense subset and every open subset of an almost $P$-space is Volterra, while weak $P$-spaces may fail to be weakly Volterra. Our example of a non-weakly Volterra weak $P$-space shows that the product of a hereditarily Baire space and a hereditarily Volterra space may fail to be weakly Volterra. Finally, we introduce the class of \emph{pseudo $P$-spaces}, a natural new weakening of $P$-spaces and construct a Hausdorff Baire pseudo $P$-space with a non-weakly Volterra subspace. The existence of a Tychonoff space with the same properties is left as an open question.

\section{$P$-spaces and generalizations}

\begin{definition} {\ \\}
\begin{enumerate}
\item A space $X$ is called a \emph{$P$-space} if every countable intersection of open subsets of $X$ is open.
\item A space $X$ is called an \emph{almost $P$-space} if every non-empty $G_\delta$ subset of $X$ has non-empty interior.
\item A space $X$ is called a \emph{weak $P$-space} if every countable subset of $X$ is closed (and discrete).
\end{enumerate}
\end{definition}

Every $P$-space is an almost $P$-space and a weak $P$-space. Examples of almost $P$-spaces which are not $P$-spaces abound. Walter Rudin \cite{R} proved that the remainder of the \v Cech-Stone compactification of the natural numbers, $\omega^*$, is an almost $P$-space. Since infinite weak $P$-spaces cannot be compact, $\omega^*$ provides an example of an almost $P$-space which is not a weak $P$-space. Steve Watson managed to construct in \cite{W} even a compact almost $P$-space where every point is the limit of a non-trivial convergent sequence.


\begin{definition} \label{defher}
Let $\mathcal{P}$ be a property of subsets of a topological space $X$. We say that $X$ is \emph{$\mathcal{P}$-hereditarily Volterra (Baire)} if every subspace of $X$ satisfying $\mathcal{P}$ is Volterra (Baire). A space is \emph{hereditarily Volterra (Baire)} if each one of its subspaces is Volterra (Baire).
\end{definition}

Contrast our Definition $\ref{defher}$ with the common habit of calling a space \emph{hereditarily Baire} if each of its closed subsets is Baire. For example, the real line is not hereditarily Baire according to our definition.

Since every subspace of a $P$-space is a $P$-space, the following proposition is clear.

\begin{proposition}
Every $P$-space is hereditarily Volterra.
\end{proposition}

\begin{proposition} \label{mainprop}
Every almost $P$-space is dense-hereditarily Volterra and open-hereditarily Volterra.
\end{proposition}

\begin{proof}
Let $X$ be an almost $P$-space. We claim that $X$ is Volterra. Indeed, let $G$ and $H$ be dense $G_\delta$ subspaces of $X$. We claim that $Int(G) \cap H$ is a dense set. Since $H$ is dense and $Int(G)$ is open we have that $\overline{Int(G) \cap H}=\overline{Int(G)}$. So if $Int(G) \cap H$ were not dense then $X \setminus \overline{Int(G)}$ would be a non-empty open set, and thus it would have to meet $G$. Therefore, $G \cap (X \setminus \overline{Int(G)})$ would be a non-empty $G_\delta$ set with empty interior. But that contradicts the fact that $X$ is an almost $P$-space.

To prove the statement of the proposition it now suffices to recall a result of R. Levy \cite{Le2} stating that every open set and every dense set of an almost $P$-space is an almost $P$-space.
\end{proof}

Almost $P$-spaces need not be hereditarily Volterra.

\begin{example}
There is a Baire regular almost $P$-space with a closed non-weakly Volterra subspace.
\end{example}

\begin{proof}
R. Levy \cite{Le1} constructed a Baire regular almost $P$-space containing a closed copy of the rational numbers, and the rational numbers are not weakly Volterra.
\end{proof}

On the other hand, weak $P$-spaces need not even be weakly Volterra. The construction of our counterexample will exploit the density topology on the real line. We recall its definition.

\begin{definition}
A measurable set $A \subset \mathbb{R}$ has density $d$ at $x$ if the limit:
$$\lim_{h \to 0} \frac{m(A \cap [x-h, x+h])}{2h}$$

exists and is equal to $d$. We denote by $d(x,A)$ the density of $A$ at $x$ and let $\phi(A)=\{x \in \mathbb{R}: d(x,A)=1 \}$.
\end{definition}

\begin{definition}
The family of all measurable sets $A \subset \mathbb{R}$ such that $\phi(A) \supset A$ defines a topology on $\mathbb{R}$ called the \emph{density topology} and denoted by $\mathbb{R}_d$.
\end{definition}

Since the density topology is finer than the Euclidean topology on the real line, every point is a $G_\delta$ set in $\mathbb{R}_d$. Moreover, every measure zero set is easily seen to be closed in $\mathbb{R}_d$. In particular, the density topology is a weak $P$-space (see \cite{T} for a comprehensive study of the density topology).

Recall that a space is \emph{resolvable} if it contains two disjoint dense sets. Dontchev, Ganster and Rose \cite{DGR} proved that the density topology is resolvable (this was later improved by Luukkainen \cite{L} who proved that $\mathbb{R}_d$ even contains a pairwise disjoint family of dense sets of size continuum). In the following lemma we review all properties of the density topology that are relevant to us here.

\begin{lemma} \label{denslemma}
$\mathbb{R}_d$ is a Tychonoff resolvable weak $P$-space with points $G_\delta$.
\end{lemma}

We also need the following lemma of Gruenhage and Lutzer.

\begin{lemma} \label{gruenlemma}
\cite{GL} Suppose $\mathcal{U}$ is a point-finite collection of open subsets of a space $X$ and that each $U \in \mathcal{U}$ contains a $G_\delta$ set $G(U)$. Then $\bigcup \{G(U): U \in \mathcal{U} \}$ is a $G_\delta$ set.
\end{lemma}

\begin{example} \label{weakexample}
There is a non-weakly Volterra Tychonoff weak $P$-space.
\end{example}

\begin{proof}
Let $X=\{f \in 2^{\omega_1}: |f^{-1}(1)|<\omega\}$ with the topology inherited from the countably supported product topology on $2^{\omega_1}$. Let $U_n=X \setminus \{f \in 2^{\omega_1}: |f^{-1}(1)| \leq n \}$, and note that $U_n$ is an open dense set in $X$. 

Use Lemma $\ref{denslemma}$ to fix disjoint dense sets $D_1$ and $D_2$ inside $\mathbb{R}_d$.

Since $\mathbb{R}_d$ is a weak $P$-space and $X$ is even a $P$-space, $X \times \mathbb{R}_d$ is a weak $P$-space. Note that the family $\{U_n \times \mathbb{R}_d: n <\omega \}$ is point-finite and $U_n \times \{x\}$ is a $G_\delta$ set contained in $U_n \times \mathbb{R}_d$ for every $x \in \mathbb{R}_d$. Thus, by Lemma $\ref{gruenlemma}$, $\bigcup_{x \in D_1} U_n \times \{x\}$ and $\bigcup_{x \in D_2} U_n \times \{x\}$ are disjoint dense $G_\delta$ sets in $X \times \mathbb{R}_d$. 
\end{proof}

Since every subspace of $\mathbb{R}_d$ is Baire (see \cite{T}), Example $\ref{weakexample}$ shows that the product of a hereditarily Volterra space and a hereditarily Baire space may fail to be weakly Volterra. This suggests the following question:

\begin{question} \label{prodquest}
Are there hereditarily Baire spaces $X$ and $Y$ such that $X \times Y$ is not weakly Volterra?
\end{question}

Note that there are metric Baire spaces whose square is not weakly Volterra (see \cite{GGP}, Example 3.9), but if an example answering $\ref{prodquest}$ in the positive exists, none of its factors can be metric. Indeed, the product of a Baire space and a closed-hereditary Baire metric space is Baire (see \cite{M}).

\section{A new weakening of $P$-spaces}

\begin{definition}
We call a space $X$ a \emph{pseudo $P$-space} if it is both an almost $P$-space and a weak $P$-space.
\end{definition}

\begin{example}
There are regular pseudo $P$-space which are not $P$-spaces.
\end{example}


\begin{proof}

For one example, let $X$ be the subspace of all weak $P$-points of $\omega^*$. In \cite{K}, Kunen proved that $X$ is a dense subset of $\omega^*$ and hence it is an almost $P$-space. Clearly $X$ is a weak $P$-space. Since there is a weak $P$-point which is not a $P$-point in $\omega^*$, $X$ is not a $P$-space though.

Another example was essentially presented in \cite{HW}. Let $X$ be a Lindel\"of $P$-space without isolated points. Van Mill (see Lemma 3.1 of \cite{HW}) proved that there is a point $p \in \beta X \setminus X$ such that $p$ is not in the closure of any countable subset of $X$. Then $X \cup \{p\}$ is a weak $P$-space. But, from the fact that $X$ is a $P$-space it follows that $X \cup \{p\}$ is an almost $P$-space. Now, $X \cup \{p\}$ is not a $P$-space, or otherwise it would be a Lindel\"of $P$-space, and thus each of its Lindel\"of subspaces should be closed. But $X$ is a non-closed Lindel\"of subspace of $X \cup \{p\}$.

\end{proof}

Pseudo $P$-spaces are in some sense very close to $P$-spaces, closer than almost $P$-spaces, so that suggests the following question. 

\begin{question}
Is there a regular pseudo $P$-space which is not hereditarily weakly Volterra?
\end{question}

The following example provides a partial answer to this question.

\begin{example} \label{pseudoexample}
There is a Hausdorff (non-regular) Baire pseudo $P$-space which is not hereditarily weakly Volterra.
\end{example}

\begin{proof}
Let $X=\{f \in 2^{\omega_1}: |f^{-1}(1)| \leq \aleph_0 \}$. Let $\mathcal{C}$ be the set of all countable partial functions from a countable subset of $\omega_1$ to $2$. For every $\sigma \in \mathcal{C}$ let $[\sigma] =\{f \in 2^{\omega_1}: \sigma \subset f \}$. Moreover, for every $n<\omega$ let $X_n=\{f \in 2^{\omega_1}: |f^{-1}(1)|=n \}$. Define a topology on $X$ by declaring $\{[\sigma] \setminus X_n: \sigma \in \mathcal{C}, n < \omega \}$ to be a subbase. 

\vspace{.1in}

\noindent {\bf Claim 1} $X$ is a pseudo $P$-space.

\begin{proof}[Proof of Claim 1]
The topology on $X$ is a refinement of the countably supported box product topology on $2^{\omega_1}$ and thus $X$ is a weak $P$-space. To prove that $X$ is an almost $P$-space, let $G=\bigcap \{U_n: n < \omega \}$ be a non-empty $G_\delta$ set and $x \in G$. For every $n<\omega$, choose $\alpha_n$ and a finite set $\mathcal{F}_n \subset \{X_k: k < \omega \}$ such that $V_n:=[x \upharpoonright \alpha_n]  \setminus \bigcup \mathcal{F}_n \subset U_n$. Let $h \in \bigcap_{n<\omega} V_n$ be a function with infinite support and $\beta<\omega_1$ be an ordinal such that $\beta \geq \sup_{n<\omega} \alpha_n$. Then $[h \upharpoonright \beta] \subset \bigcap_{n<\omega} V_n \subset \bigcap_{n<\omega} U_n$.
\renewcommand{\qedsymbol}{$\triangle$}
\end{proof}

\noindent {\bf Claim 2} The space $X$ is Baire.

\begin{proof}[Proof of Claim 2] 
We prove that every meager set is nowhere dense. Let $\{N_n: n < \omega \}$ be a countable family of nowhere dense subsets of $X$. Let $\sigma$ be a countable partial function with domain $\alpha < \omega_1$ and $k$ be an integer: we are going to prove that the basic open set $[\sigma] \setminus \bigcup \{X_n: n \leq k\}$ is not contained in the closure of $\bigcup_{n<\omega} N_n$. Since $N_0$ is nowhere dense there must be a countable partial function $\sigma_0$ extending $\sigma$ with domain $\alpha_0 > \alpha$ and  an integer $k_0 < \omega$ such that $([\sigma_0] \setminus \bigcup \{X_k: k \leq k_0\}) \cap N_0=\emptyset$.

Suppose we've found an increasing sequence of countable partial functions $\{\sigma_i: i < n \}$ and an increasing sequence of integers $\{k_i: i < n \}$. Since $N_n$ is nowhere dense there must be a countable partial function $\sigma_n$ extending $\sigma_{n-1}$ and an integer $k_n>k_{n-1}$ such that $[\sigma_n] \cap N_n=\emptyset$. Let $\sigma_\omega=\bigcup_{i<\omega} \sigma_i$. Then $([\sigma_\omega] \setminus \bigcup \{X_k: k < \omega \}) \cap \bigcup_{n<\omega} N_n=\emptyset$ and $\emptyset \neq [\sigma_\omega] \subset ([\sigma] \setminus \bigcup \{X_n: n \leq k \})$. Thus $[\sigma] \setminus \bigcup \{X_n: n \leq k \}$ is not contained in $\overline{\bigcup_{n<\omega} N_n}$ and since the choice of $\sigma$ and $k$ was arbitrary, this shows that $\bigcup_{n<\omega} N_n$ is nowhere dense.
\renewcommand{\qedsymbol}{$\triangle$}
\end{proof}

\noindent {\bf Claim 3} Let $Y=\bigcup_{n<\omega} X_n \subset X$. Then $Y$ is not weakly Volterra.

\begin{proof}[Proof of Claim 3]
Let $G=\bigcap \{X \setminus X_k: k$ is even $\}$ and $H=\bigcap \{X \setminus X_k: k$ is odd $\}$. Then $G$ and $H$ are dense $G_\delta$ subsets of $Y$ with empty intersection.
\renewcommand{\qedsymbol}{$\triangle$}
\end{proof}
\end{proof}

As pointed out by Gary Gruenhage, Example $\ref{pseudoexample}$ is not regular. For example, the closed set $X_1$ and the null function cannot be separated by disjoint open sets.



\begin{thebibliography}{10}
\bibitem{B} R. Baire, \emph{Sur les fonctions de variables r\' eelles}, Ann. di Mat., \textbf{3} (1899), 1--123.
\bibitem{CG} J. Cao and D. Gauld, \emph{Volterra spaces revisited}, J. Aust. Math. Soc. \textbf{79} (2005), 61--76.
\bibitem{CJ} J. Cao and H. Junnila, \emph{When is a Volterra space Baire?}, Topology Appl. \textbf{154} (2007), 527--532.
\bibitem{DGR} J. Dontchev, M. Ganster and D. Rose, \emph{Ideal Resolvability}, Topology Appl. \textbf{93} (1999), 1--16
\bibitem{D} W. Dunham, \emph{A historical gem from Vito Volterra}, Mathematics Magazine, \textbf{63} (1990), 234--237.
\bibitem{GGP} D. Gauld, S. Greenwood and Z. Piotrowski, \emph{On Volterra spaces III: Topological operations} Topology Proc. \textbf{23} (1998), 167--182. 
\bibitem{GP} D. Gauld and Z. Piotrowski, \emph{On Volterra spaces}, Far East J. Math. Sci. \textbf{1} (1993), 209--214.
\bibitem{GL} G. Gruenhage and D. Lutzer, \emph{Baire and Volterra spaces}, Proc. Amer. Math. Soc. \textbf{128} (2000), 3115--3124.
\bibitem{HW} M. Henriksen and R.G. Woods, \emph{Weak $P$-spaces and $L$-closed spaces}, Q \& A in General Topology \textbf{6} (1988), 201--207.
\bibitem{K} K. Kunen, \emph{Weak $P$-points in $\mathbb{N}^*$}, Colloq. Math. Soc. J\'anos Bolyai, vol. 23, Topology (Budapest, 1978), 741--749.
\bibitem{Le1} R. Levy, \emph{Showering Spaces}, Pacific J. Math. \textbf{57} (1975), 223--232. 
\bibitem{Le2} R. Levy, \emph{Almost $P$-spaces}, Canad. J. Math. \textbf{29} (1977), 284--288.
\bibitem{L} J. Luukkainen, \emph{The density topology is maximally resolvable}, Real Anal. Exchange \textbf{25} (1999), 419--420. 
\bibitem{M} W. Moors, \emph{The product of a Baire space with a hereditarily Baire metric space is Baire}, Proc. Amer. Math. Soc. \textbf{134} (2006), 2161--2163
\bibitem{R} W. Rudin, \emph{Homogeneity problems in the theory of Cech compactifications}, Duke Math. J. \textbf{23} (1956), 409--419.
\bibitem{T} F.D. Tall, \emph{The Density Topology}, Pacific J. Math. \textbf{62} (1976), 275--284.
\bibitem{V} V. Volterra, \emph{Alcune osservasioni sulle funzioni punteggiate discontinue}, Giornale di Matematiche \textbf{19} (1881), 76--86.
\bibitem{W} S. Watson, \emph{A compact Hausdorff space without $P$-points in which $G_\delta$ sets have interior}, Proc. Amer. Math. Soc. \textbf{123} (1995), 2575--2577.
\end{thebibliography}
\end{document}